\documentclass[leqno]{amsart}

\usepackage{amsmath,amssymb, latexsym, amscd, amsthm,
  amsfonts,amstext,graphicx,color}

\theoremstyle{plain}
\newtheorem{thm}{Theorem}[section]

\newtheorem{prop}[thm]{Proposition}

\theoremstyle{definition}
\newtheorem{defn}[thm]{Definition}

\theoremstyle{remark}
\newtheorem{rem}[thm]{Remark}

\newtheorem{example}[thm]{Example}
\numberwithin{equation}{section}

\numberwithin{equation}{section}

\newcommand{\R}{\mathbb{R}}

\newcommand{\PP}{\mathbb{P}}

\newcommand{\E}{\mathbb{E}}

\newcommand{\x}{\mathbf{x}}

\newcommand{\p}{\mathbf{p}}

\newcommand{\bel}[1]{\begin{equation}\label{#1}}

\newcommand{\be}{\begin{equation}}
\newcommand{\ba}{\begin{eqnarray}}
\newcommand{\ea}{\end{eqnarray}}

\newcommand{\qe}{\end{equation}}
\newcommand{\al}{\alpha}
\newcommand{\alv}{\boldsymbol{\alpha}}

\newcommand{\ld}{\lambda}
\newcommand{\de}{\delta}

\newcommand{\suml}{\sum\limits}
\newcommand{\intl}{\int\limits}
\newcommand{\prodl}{\prod\limits}
\newcommand{\ble}{\left\{}
\newcommand{\bri}{\right\}}
\newcommand{\ple}{\left (}
\newcommand{\pri}{\right )}

\newcommand{\ov}{\boldsymbol{0}}

\begin{document}
\title[Wright-Fisher model]{A general solution
  of the Wright-Fisher model of random genetic drift}
\author{Tat Dat Tran, Julian Hofrichter, J\"{u}rgen Jost}
\address{Tat Dat Tran, Max Planck Institute for Mathematics in the Sciences, D-04103 Leipzig, Germany}
\email{trandat@mis.mpg.de}
\address{Julian Hofrichter, Max Planck Institute for Mathematics in the Sciences, D-04103 Leipzig, Germany}
\email{julian.hofrichter@mis.mpg.de}
\address{J\"{u}rgen Jost, Max Planck Institute for Mathematics in the
  Sciences, D-04103 Leipzig, Germany\\
Department of Mathematics, Leipzig University, D-04081 Leipzig,
Germany\\
Santa Fe Institute for the Sciences of Complexity, Santa Fe, NM 87501,
USA}
\email{jost@mis.mpg.de}
\thanks{
The research leading
to these results has received funding from the European Research
Council under the European Union's Seventh Framework Programme
(FP7/2007-2013) / ERC grant agreement no. 267087. The first and the  second
author have also been supported by the IMPRS ``Mathematics in the Sciences'', and the material presented in this
paper is largely based on the first author's thesis.}
\date{\today}
\maketitle
\begin{abstract}
We develop a general solution for the Fokker-Planck (Kolomogorov) equation
representing the diffusion limit of the Wright-Fisher model of random
genetic drift for an arbitrary number of alleles at a single
locus. From this solution, we can readily deduce  information about
the evolution of a Wright-Fisher population.
\end{abstract}
\section{Introduction}
The random genetic drift model developed implicitly by Fisher in
\cite{fisher} and explicitly by Wright in \cite{wright1}, and
henceforth called the  Wright-Fisher model is one of the most popular
stochastic models in population genetics (\cite{ewens,buerger}). In
its simplest form, it is concerned with the evolution of the
probabilities between non-overlapping generations in a population of
fixed size of two alleles at a single diploid locus that are obtained from
random sampling in the parental generation, without additional
biological mechanisms like mutation, selection, or a spatial
population structure. Generalizations to multiple alleles, several
loci, inclusion of mutations and selection etc. then constituted an
important part of mathematical population genetics. It is our aim to
develop a general mathematical perspective on the Wright-Fisher model
and its generalizations. In the present paper, we treat the case of
multiple alleles at a single site. In a companion paper \cite{thj1}, we discuss
the simplest case of 2 alleles in more detail. Generalizations will be addressed in
subsequent papers.

Let us first describe the basic mathematical contributions of Wright
and Kimura.  In
1945, Wright  approximated the discrete process by a  diffusion
processthat is continuous in space and time  (continuous process, for short)
and that can be described by a Fokker-Planck equation. In 1955, by solving
this Fokker-Planck equation derived from the Wright-Fisher model,
Kimura obtained an exact solution for the Wright-Fisher model in the
case of $2$ alleles (see \cite{kimura1}). Kimura (\cite{kimura2}) also
developed an approximation for the solution of the Wright-Fisher model
in the multi-allele case, and in 1956, he obtained (\cite{kimura3}) a
exact solution of this model for $3$ alleles and concluded that this
can be generalized to arbitrarily many alleles. This yields more
information about the Wright-Fisher model as well as the corresponding
continuous process. Kimura's solution, however, is not entirely
satisfactory. For one thing, it depends on very clever algebraic
manipulations so that the general mathematical structure is not very
transparent, and this makes generalizations very difficult. Also,
Kimura's approach is local in the sense that it does not naturally
incorporate the transitions resulting from the (irreversible) loss of
one or more alleles in the population. Therefore, for instance the
integral of his probability density function on its defined domain is
not equal to $1$.

In this paper, we derive the formalism for a general solution that
naturally includes the transitions resulting from the disappearance of
alleles. The key are evolution equations for the moments of the
probability density. We show that there exists a unique global
solution of the Fokker-Planck equation. We then utilize this solution
to derive properties of the underlying process, like the expected
transition times.

\section{The global solution of the Wright-Fisher model}
In this section, we will establish some notations, and then prove some
propositions as well as the main theorem of this paper.
\subsection{Notations} $\Delta_n:=\{(x^1,x^2,\dots ,x^{n+1}):
\sum_{i=1}^{n+1} x^i=1\}$ is the standard $n-$simplex in
$\R^{n+1}$ representing the probabilities or relative frequencies of alleles $A_1,\dots
, A_{n+1}$ in our population. Often, however, it is advantageous to work in $\R^n$
instead of $\R^{n+1}$,
and with $e_0:= (0,\ldots,0) \in \R^n,\:
e_k:=(0,\ldots,\underbrace{1}_{k^{th}},\ldots,0)\in \R^n$, we
therefore define
\begin{equation*}
\Omega_n :=\: \mathrm{intco} \ble e_0,\ldots,e_n \bri := \ble \suml_{k=0}^n x^k e_k,\: (x,x^0)=\ple x^1,\ldots,x^n,1-\suml_{k=1}^n x^k\pri \in \mathrm{int} \Delta_n \bri  .
\end{equation*}
Moreover, we shall need the subsimplices corresponding to subsets of
alleles, using the following notations
\begin{equation*}
\begin{split}
I_k:=&\ble \ble i_0,\ldots,i_k \bri,\: 0\le i_0<\ldots <i_k \le n\bri, \quad k\in \{1,\ldots,n\},\\
V_0 :=&\: \ble e_0,\ldots, e_n \bri,\\
& \text{ the domain representing a population of one allele},\\
V_k^{(i_0,\ldots,i_k)}:=&\: \mathrm{intco} \ble e_{i_0},\ldots,e_{i_k}\bri, \quad k\in \{1,\ldots,n\},\\
& \text{ the domain representing a population of alleles $\{A_{i_0},\ldots,A_{i_k}\}$,}\\
V_k :=&\:\ble \mathrm{intco} \ble e_{i_0},\ldots,e_{i_k}\bri \text{ for some } i_0<\ldots<i_k \in \overline{0,n} \bri, \quad k\in \{1,\ldots,n\},\\
=& \bigsqcup\limits_{(i_0,\ldots,i_k)\in I_k} V_k^{(i_0,\ldots,i_k)},\\
& \text{ the domain representing a population of $(k+1)$ alleles,}\\
\overline{V}_k:=&\:\bigcup\limits_{(i_0,\ldots,i_k)\in I_k} \overline{V}_k^{(i_0,\cdots,i_k)}, \quad k\in \{1,\ldots,n\},\\
=&\: \bigsqcup\limits_{i=0}^k V_i,\\
& \text{ the domain representing a population of at most $(k+1)$
  alleles.}
\end{split}
\end{equation*}
We shall also need some function spaces:
\begin{equation*}
\begin{split}
H_k^{(i_0,\ldots, i_k)}:=&\: C^\infty \ple \overline{V_k^{(i_0,\ldots, i_k)}} \pri,\\
H_k:=&\: C^\infty(\overline{V}_k),\quad k\in \{1,\ldots,n\},\\
H:=&\: \ble f:\overline{V}_n\to [0,\infty] \text{ measurable such that } [f,g]_n<\infty, \forall g\in H_n\bri,\\
&\: \text{where } [f,g]_n:=\:\intl_{\overline{V}_n}f(x)g(x)d\mu(x)=\suml_{k=0}^n\intl_{V_k} f(x)g(x)d\mu_k(x),\\
&\:\quad\quad\quad\quad\quad =\suml_{k=0}^n\suml_{(i_0,\ldots,i_k)\in I_k}\intl_{V^{(i_0,\ldots,i_k)}_k} f(x)g(x)d\mu^{(i_0,\ldots,i_k)}_k(x),\\
&\: \text{with $\mu^{(i_0,\ldots,i_k)}_k$  a probability measure on
  $V^{(i_0,\ldots,i_k)}_k$}.
\end{split}
\end{equation*}
We can now define the differential operators for our Fokker-Planck equation:
\begin{equation*}
\begin{split}
L_k^{(i_0,\ldots,i_k)}:&\: H_k^{(i_0,\ldots, i_k)}\to H_k^{(i_0,\ldots, i_k)},\: L_k^{(i_0,\ldots, i_k)}f(x)=\frac{1}{2}\suml_{i,j\in \{i_1,\ldots, i_k \}} \frac{\partial^2 (a_{ij}(x)f(x))}{\partial x^i \partial x^j} ,\\
(L_k^{(i_0,\ldots, i_k)})^{*}:&\: H_k^{(i_0,\ldots, i_k)}\to H_k^{(i_0,\ldots, i_k)},\: (L_k^{(i_0,\ldots, i_k)})^*g(x)=\frac{1}{2} \suml_{i,j\in \{i_1,\ldots, i_k \}} a_{ij}(x) \frac{\partial^2 g(x)}{\partial x^i \partial x^j},\\
L_k:&\: H_k\to H_k,\quad (L_k)_{|{H_k^{(i_0,\ldots, i_k)}}} = L_k^{(i_0,\ldots,i_k)},\\
L_k^*:&\: H_k\to H_k,\quad (L_k^*)_{|{H_k^{(i_0,\ldots, i_k)}}} = (L_k^{(i_0,\ldots, i_k)})^{*},
\end{split}
\end{equation*}
where the coefficients are defined by
\begin{equation*}
a_{ij}(x):=\: x^i(\de_{ij}-x^j),\quad i,j\in \{1,\ldots,n\}.
\end{equation*}
Finally, we shall need
\begin{equation*}
w_k^{(i_0,\ldots, i_k)}(x):=\prodl_{i\in I_k^{(i_0,\ldots, i_k)}} x^i, \quad k\in \{1,\ldots,n\}.
\end{equation*}

\begin{prop}\label{prop:1}
For each $1\le k\le n$, $m\ge 0, |\al|=\al^1+\cdots+\al^k=m$, the
 polynomial of degree $m$ in $k$ variables $x=(x^{i_1},\ldots,x^{i_k})$ in $\overline{V_k^{(i_0,\ldots,i_k)}}$
\bel{poly}
X_{m,\al}^{(k)}(x)=x^\al+\suml_{|\beta|<m}a^{(k)}_{m,\beta} x^\beta,
\qe
where the $a^{(k)}_{m,\beta}$ are inductively defined by
$$
a^{(k)}_{m,\beta}=-\frac{\suml_{i=1}^k (\beta_i+2)(\beta_i+1)a^{(k)}_{m,\beta+e_i}}{(m-|\beta|)(m+\beta+2k+1)},\quad \forall |\beta|<m,
$$
is the eigenvector of $L_k^{(i_0,\ldots,i_k)}$ corresponding to the
eigenvalue $\ld^{(k)}_m=\frac{(m+k)(m+k+1)}{2}$.
\end{prop}

\begin{proof}
We have
\begin{equation*}
\begin{split}
L_k^{(i_0,\ldots,i_k)} X_{m,\al}^{(k)}(x)& =\frac{1}{2}\sum_{i \in \{i_1,\ldots,i_k\}} \frac{\partial ^2}{(\partial x^i)^2}\Bigg[x^i(1-x^i)\Big(x^\al+\sum_{|\beta|<m}a_{m,\beta}^{(k)}x^\beta\Big)\Bigg]\\
&\quad -\sum_{i\ne j \in \{i_1,\ldots,i_k\}} \frac{\partial ^2}{\partial x^i \partial x^j}\Bigg[x^ix^j\Big(x^\al+\sum_{|\beta|<m}a_{m,\beta}^{(k)}x^\beta\Big)\Bigg]\\
& = \frac{1}{2}\sum_{i \in \{i_1,\ldots,i_k\}} \frac{\partial ^2}{(\partial x^i)^2}\Bigg[x^{\al+e_i}-x^{\al+2e_i}+\sum_{|\beta|<m}a_{m,\beta}^{(k)}x^{\beta+e_i}\\
&\quad \quad -\sum_{|\beta|<m}a_{m,\beta}^{(k)}x^{\beta+2e_i}\Bigg]\\
&\quad -\sum_{i\ne j \in \{i_1,\ldots,i_k\}} \frac{\partial ^2}{\partial x^i \partial x^j}\Bigg[x^{\al+e_i+e_j}+\sum_{|\beta|<m}a_{m,\beta}^{(k)}x^{\beta+e_i+e_j}\Bigg]\\
& = \frac{1}{2}\sum_{i} \Bigg[(\al^i+1)\al^i x^{\al-e_i}-(\al^i+2)(\al^i+1)x^{\al}\\
&\quad \quad +\sum_{|\beta|<m}a_{m,\beta}^{(k)}(\beta^i+1)\beta^i x^{\beta-e_i}-\sum_{|\beta|<m}a_{m,\beta}^{(k)}(\beta^i+2)(\beta^i+1)x^{\beta}\Bigg]\\
&\quad -\sum_{i\ne j} \Bigg[(\al^i+1)(\al^j+1)x^{\al}+\sum_{|\beta|<m}a_{m,\beta}^{(k)}(\beta^i+1)(\beta^j+1)x^{\beta}\Bigg]\\
& = \Bigg[-\frac{1}{2}\sum_{i} (\al^i+2)(\al^i+1)-\sum_{i\ne j}(\al^i+1)(\al^j+1)\Bigg]x^\al\\
&\quad + \text{terms of lower degree}\\
& = \Bigg[-\frac{1}{2}\Big(\sum_{i} \al^i+k\Big)\Big(\sum_{i}\al^i+k+1\Big)\Bigg]x^\al + \text{terms of lower degree}\\
& = -\frac{(m+k)(m+k+1)}{2}x^\al + \text{terms of lower degree}.
\end{split}
\end{equation*}
By equalizing coefficients we obtain
$$
\ld^{(k)}_m=\frac{(m+k)(m+k+1)}{2}
$$
and
$$
a^{(k)}_{m,\beta}=-\frac{\suml_{i=1}^k (\beta_i+2)(\beta_i+1)a^{(k)}_{m,\beta+e_i}}{(m-|\beta|)(m+\beta+2k+1)},\quad \forall |\beta|<m.
$$
This completes the proof.
\end{proof}

\begin{rem}
When $k=1$,  $X^{(1)}_{m,m}(x^1)$ is the $m^{th}-$Gegenbauer
polynomial (up to a constant). Thus, the polynomials
$X_{m,\al}^{(k)}(x)$  can be understood as a generalization of the
Gegenbauer polynomials to higher dimensions.
\end{rem}

\begin{prop}\label{prop:2}
If $X\in \overline{V_k^{(i_0,\ldots,i_k)}}$ is an eigenvector of $L_k^{(i_0,\ldots,i_k)}$ corresponding to $\ld$ then $w_k^{(i_0,\ldots,i_k)}X$ is  an eigenvector of $(L_k^{(i_0,\ldots,i_k)})^*$ corresponding to $\ld$.
\end{prop}
\begin{proof}
If $X\in \overline{V_k^{(i_0,\ldots,i_k)}}$ is an eigenvector of $L_k^{(i_0,\ldots,i_k)}$ corresponding to $\ld$, it follows that
\begin{equation*}
\begin{split}
-\ld (w_k^{(i_0,\ldots,i_k)}(x) X)=& \frac{1}{2} w_k^{(i_0,\ldots,i_k)}(x) \suml_{i,j\in \{i_1,\ldots, i_k \}} \frac{\partial^2}{\partial x^i \partial x^j}\ple x^i(\de_{ij}-x^j)X\pri\\
=&\frac{1}{2} w_k^{(i_0,\ldots,i_k)}(x) \suml_{i,j\in \{i_1,\ldots, i_k \}} \ple x^i(\de_{ij}-x^j)\pri\frac{\partial^2 X}{\partial x^i \partial x^j}\\
&+\frac{1}{2} w_k^{(i_0,\ldots,i_k)}(x) \suml_{i,j\in \{i_1,\ldots, i_k \}} \frac{\partial\ple x^i(\de_{ij}-x^j)\pri}{\partial x^i}\frac{\partial X}{\partial x^j}\\
&+\frac{1}{2} w_k^{(i_0,\ldots,i_k)}(x) \suml_{i,j=1}^k \frac{\partial\ple x^i(\de_{ij}-x^j)\pri}{\partial x^j}\frac{\partial X}{\partial x^i}\\
&+\frac{1}{2} w_k^{(i_0,\ldots,i_k)}(x) \suml_{i,j\in \{i_1,\ldots, i_k \}} \frac{\partial^2 \ple x^i(\de_{ij}-x^j)\pri}{\partial x^i\partial x^j}X\\
=&\frac{1}{2}\suml_{i,j=1}^k \ple x^i(\de_{ij}-x^j)\pri \ple w_k^{(i_0,\ldots,i_k)}(x)\frac{\partial^2 X}{\partial x^i \partial x^j}\pri \\
&+ \frac{1}{2}\suml_{j\in \{i_1,\ldots, i_k \}} w_k^{(i_0,\ldots,i_k)}(x)\ple 1-(k-1)x^j\pri \frac{\partial X}{\partial x^j}\\
&+ \frac{1}{2}\suml_{i\in \{i_1,\ldots, i_k \}} w_k^{(i_0,\ldots,i_k)}(x)\ple 1-(k-1)x^i\pri \frac{\partial X}{\partial x^i}\\
&-\frac{k(k+1)}{2}w_k^{(i_0,\ldots,i_k)}(x) X\\
=&\frac{1}{2}\suml_{i,j\in \{i_1,\ldots, i_k \}} \ple x^i(\de_{ij}-x^j)\pri \ple w_k^{(i_0,\ldots,i_k)}(x)\frac{\partial^2 X}{\partial x^i \partial x^j}\pri\\
&+\frac{1}{2}\suml_{i,j\in \{i_1,\ldots, i_k \}} \ple x^i(\de_{ij}-x^j)\pri \frac{\partial w_k^{(i_0,\ldots,i_k)}(x)}{\partial x^i} \frac{\partial X}{\partial x^j}\\
&+\frac{1}{2}\suml_{i,j\in \{i_1,\ldots, i_k \}} \ple x^i(\de_{ij}-x^j)\pri \frac{\partial w_k^{(i_0,\ldots,i_k)}(x)}{\partial x^j} \frac{\partial X}{\partial x^i}\\
&+\frac{1}{2}\suml_{i,j\in \{i_1,\ldots, i_k \}} \ple x^i(\de_{ij}-x^j)\pri \frac{\partial^2 w_k^{(i_0,\ldots,i_k)}(x)}{\partial x^i \partial x^j} X\\
=&\frac{1}{2}\suml_{i,j\in \{i_1,\ldots, i_k \}} \ple x^i(\de_{ij}-x^j)\pri \frac{\partial ^2 (w_k^{(i_0,\ldots,i_k)} X)(x)}{\partial x^i \partial x^j}\\
=&(L_k^{(i_0,\ldots,i_k)})^* (w_k^{(i_0,\ldots,i_k)}(x) X).
\end{split}
\end{equation*}
This completes the proof.
\end{proof}

\begin{prop}\label{prop:3}
Let $\nu$ be the exterior unit normal vector of the domain $V_k^{(i_0,\ldots,i_k)}$. Then we have
\bel{eq:condbound}
\suml_{j\in \{i_1,\ldots, i_k \}} a_{ij}\nu^j =0 \quad \text{ on } \partial V_k^{(i_0,\ldots,i_k)},\quad \forall i\in \{i_1,\ldots, i_k \}.
\qe
\end{prop}
\begin{proof}
In fact, on the surface $(x^{s}=0)$, for some $s\in \{i_1,\ldots, i_k
\}$ we have $\nu=-e_s$, and hence $\suml_{j\in \{i_1,\ldots, i_k \}} a_{ij}\nu^j=a_{is}=x^s(\de_{si}-x^i)=0$. On the surface $(x^{i_0}=0)$ we have $\nu=\frac{1}{\sqrt{k}}(e_{i_1}+\ldots+e_{i_k})$, hence $\suml_{j\in \{i_1,\ldots, i_k \}} a_{ij}\nu^j=\frac{1}{\sqrt{k}}\suml_{j\in \{i_1,\ldots, i_k \}} a_{ij}=\frac{1}{\sqrt{k}}x^i x^{i_0} =0$. This completes the proof.
\end{proof}

\begin{prop}\label{prop:4}
$L_k^{(i_0,\ldots,i_k)}$ and $(L_k^{(i_0,\ldots,i_k)})^*$ are weighted adjoints in $H_k^{(i_0,\ldots, i_k)}$, i.e.
$$
(L_k^{(i_0,\ldots,i_k)} X,w_k^{(i_0,\ldots,i_k)} Y)=(X,(L_k^{(i_0,\ldots,i_k)})^*(w_k^{(i_0,\ldots,i_k)} Y)),\quad \forall X,Y \in H_k^{(i_0,\ldots, i_k)}.
$$
\end{prop}
\begin{proof} We put $F^{(k)}_i(x):=\suml_{j\in \{i_1,\ldots, i_k \}}
  \frac{\partial (a_{ij}(x)X(x))}{\partial x^j}$. Because of
  $w_k^{(i_0,\ldots,i_k)} Y \in
  C^\infty_0(\overline{V}_k^{(i_0,\ldots,i_k)})$, the second Green formula, and Proposition \ref{prop:3}, we have
\begin{equation*}
\begin{split}
(L_k^{(i_0,\ldots,i_k)} X,w_k^{(i_0,\ldots,i_k)} Y)=& \frac{1}{2}\suml_{i,j\in \{i_1,\ldots, i_k \}} \intl_{\overline{V}_k^{(i_0,\ldots,i_k)}} \frac{\partial^2 (a_{ij}(x)X(x))}{\partial x^i \partial x^j} w_k^{(i_0,\ldots,i_k)}(x)Y(x)dx\\
=&\frac{1}{2}\suml_{i\in \{i_1,\ldots, i_k \}} \intl_{\overline{V}_k^{(i_0,\ldots,i_k)}} \frac{\partial F^{(k)}_i(x)}{\partial x^i} w_k^{(i_0,\ldots,i_k)}(x)Y(x)dx\\
=&\frac{1}{2}\suml_{i\in \{i_1,\ldots, i_k \}} \intl_{\partial V_k^{(i_0,\ldots,i_k)}} F^{(k)}_i(x) \nu_i w_k^{(i_0,\ldots,i_k)}(x)Y(x)do(x)\\
& -\frac{1}{2}\suml_{i\in \{i_1,\ldots, i_k \}} \intl_{\overline{V}_k^{(i_0,\ldots,i_k)}} F^{(k)}_i(x) \frac{\partial (w_k^{(i_0,\ldots,i_k)}(x)Y(x))}{\partial x^i} dx\\
=&-\frac{1}{2}\suml_{i\in \{i_1,\ldots, i_k \}} \intl_{\overline{V}_k^{(i_0,\ldots,i_k)}} F^{(k)}_i(x) \frac{\partial (w_k^{(i_0,\ldots,i_k)}(x)Y(x))}{\partial x^i} dx\\
=&-\frac{1}{2}\suml_{i,j\in \{i_1,\ldots, i_k \}} \intl_{\overline{V}_k^{(i_0,\ldots,i_k)}} \frac{\partial (a_{ij}(x)X(x))}{\partial x^j} \frac{\partial (w_k^{(i_0,\ldots,i_k)}(x)Y(x))}{\partial x^i} dx\\
=&-\frac{1}{2}\suml_{i,j\in \{i_1,\ldots, i_k \}} \intl_{\partial V_k^{(i_0,\ldots,i_k)}} a_{ij}(x)\nu_j X(x)\frac{\partial (w_k^{(i_0,\ldots,i_k)}(x)Y(x))}{\partial x^i} do(x)\\
&+\ple X,L_k^*(w_k^{(i_0,\ldots,i_k)} Y)\pri\\
=&\ple X,L_k^*(w_k Y)\pri.
\end{split}
\end{equation*}

\end{proof}

\begin{prop}\label{prop:5}
In $\overline{V}_k^{(i_0,\ldots,i_k)}$,  $\ble X^{(k)}_{m,\al}
\bri_{m\ge 0,|\al|=m}$ is a basis of $H_k^{(i_0,\ldots, i_k)}$ which
is orthogonal with respect to the weights $w_k^{(i_0,\ldots,i_k)}$, i.e.,
$$
\ple X^{(k)}_{m,\al}, w_k^{(i_0,\ldots,i_k)} X^{(k)}_{j,\beta}\pri = 0, \quad \forall j\ne m, |\al|=m, |\beta|=j.
$$
\end{prop}
\begin{proof}
$\ble X^{(k)}_{m,\al}\bri_{m\ge 0,|\al|=m}$ is a basis of
$H_k^{(i_0,\ldots, i_k)}$ because $\ble x^\al\bri_{\al}$ is a basis of
this space. To prove the orthogonality we apply the Propositions \ref{prop:1}, \ref{prop:2}, \ref{prop:4} as follows
\begin{equation*}
\begin{split}
-\ld^{(k)}_{m} \ple X^{(k)}_{m,\al}, w_k^{(i_0,\ldots,i_k)} X^{(k)}_{j,\beta}\pri
=& \ple L_k^{(i_0,\ldots,i_k)} X^{(k)}_{m,\al}, w_k^{(i_0,\ldots,i_k)} X^{(k)}_{j,\beta} \pri\\
=& \ple X^{(k)}_{m,\al}, (L_k^{(i_0,\ldots,i_k)})^*(w_k^{(i_0,\ldots,i_k)} X^{(k)}_{j,\beta})\pri\\
=&-\ld^{(k)}_{j}\ple X^{(k)}_{m,\al}, w_k^{(i_0,\ldots,i_k)} X^{(k)}_{j,\beta}\pri
\end{split}
\end{equation*}
Because $\ld^{(k)}_{m}\ne \ld^{(k)}_{j}$, this finishes the proof.
\end{proof}

\begin{prop}\label{prop:4}
\begin{enumerate}
\item[(i)] The spectrum of the operator $L_k^{(i_0,\ldots,i_k)}$ is
$$
Spec(L_k^{(i_0,\ldots,i_k)})=\bigcup_{m\ge 0} \ble \ld_m^{(k)}=\frac{(m+k)(m+k+1)}{2}\bri=:\Lambda_k
$$
and the eigenvectors of $L_k^{(i_0,\ldots,i_k)}$ corresponding to
$\ld^{(k)}_m$ are of the form
$$
X=\suml_{|\al|=m} d^{(k)}_{m,\al} X^{(k)}_{m,\al},
$$
i.e., the eigenspace corresponding to $\ld^{(k)}_m$ are of dimension $k+m-1 \choose{k-1}$;
\item[(ii)] The spectrum of the operator $L_k$ is the same.
\end{enumerate}
\begin{proof}
\begin{enumerate}
\item[(i)] Proposition \ref{prop:1} implies that $\Lambda_k\subseteq
  Spec(L_k^{(i_0,\ldots,i_k)})$. Conversely, for $\ld \notin
  \Lambda_k$, we will prove that $\ld$  is not an eigenvalue of
  $L_k^{(i_0,\ldots,i_k)}$. In fact, assume that $X \in
  H_k^{(i_0,\ldots, i_k)}$ such that $L_k^{(i_0,\ldots,i_k)} X= -\ld
  X$ in $H_k^{(i_0,\ldots, i_k)}$. Because $\ble
  X^{(k)}_{m,\al}\bri_{m,\al}$ is an orthogonal basis of
  $H_k^{(i_0,\ldots, i_k)}$ with respect to the weights
  $w_k^{(i_0,\ldots,i_k)}$ (Proposition \ref{prop:3}), we can
  represent $X$ by $X=\suml_{m=0}^\infty\suml_{|\al|=m}
  d^{(k)}_{m,\al} X^{(k)}_{m,\al}$. It follows that
\begin{equation*}
\begin{split}
\suml_{m=0}^\infty \suml_{|\al|=m} d^{(k)}_{m,\al} (-\ld^{(k)}_m)X^{(n)}_{m,\al}=&\suml_{m=0}^\infty \suml_{|\al|=m} d^{(k)}_{m,\al} L_k^{(i_0,\ldots,i_k)} X^{(k)}_{m,\al}\\
=&L_k^{(i_0,\ldots,i_k)} X\\
=&-\ld \suml_{m=0}^\infty \suml_{|\al|=m} d^{(k)}_{m,\al} X^{(k)}_{m,\al}.
\end{split}
\end{equation*}
For any $j\ge 0$, $|\beta|=j$, multiplying  by $w_k X^{(k)}_{j,\beta}$ and then integrating on $\overline{V}_n$ we have
\begin{equation*}
\begin{split}
& \suml_{|\al|=j} d^{(k)}_{j,\al} \ld^{(k)}_j\ple X^{(k)}_{j,\al},w_k^{(i_0,\ldots,i_k)} X^{(k)}_{j,\beta}\pri = \suml_{|\al|=j} d^{(k)}_{j,\al} \ld \ple X^{(k)}_{j,\al},w_k^{(i_0,\ldots,i_k)} X^{(k)}_{j,\beta}\pri, \forall j\ge 0, |\beta|=j,\\
\Rightarrow & \ple X^{(k)}_{j,\al},w_k^{(i_0,\ldots,i_k)} X^{(k)}_{j,\beta}\pri_{\beta,\al} (d^{(k)}_{j,\al} \ld^{(k)}_j)_{\al}=\ple X^{(k)}_{j,\al},w_k^{(i_0,\ldots,i_k)} X^{(k)}_{j,\beta}\pri_{\beta,\al} (d^{(k)}_{j,\al} \ld)_{\al}, \forall j\ge 0, |\beta|=j,\\
\Rightarrow & d^{(k)}_{j,\al} \ld^{(k)}_j=d^{(k)}_{j,\al} \ld,\quad \forall j\ge 0, |\beta|=j, \text{ because } det \ple X^{(k)}_{j,\al},w_k^{(i_0,\ldots,i_k)} X^{(k)}_{j,\beta}\pri_{\beta,\al} \ne 0\\
\Rightarrow & d^{(k)}_{j,\al} = 0,\quad \forall j\ge 0, |\al|=j, \text{ because } \ld \ne \ld^{(k)}_j.
\end{split}
\end{equation*}
It follows that $X=0$ in $H_k^{(i_0,\ldots, i_k)}$. Therefore
$$
Spec(L_k^{(i_0,\ldots,i_k)})=\bigcup_{m\ge 0} \ble \ld_m^{(k)}=\frac{(m+k)(m+k+1)}{2}\bri=\Lambda_k.
$$
Moreover, assume that $X\in H_k^{(i_0,\ldots, i_k)}$ is an eigenvector of $L_k^{(i_0,\ldots,i_k)}$ corresponding to $\ld^{(k)}_j$, i.e., $L_k^{(i_0,\ldots,i_k)} X=-\ld_j X$. We represent $X$ by
$$
X=\suml_{m=0}^\infty\suml_{|\al|=m} d^{(k)}_{m,\al} X^{(k)}_{m,\al}.
$$
It follows that
\begin{equation*}
\begin{split}
\suml_{m=0}^\infty \suml_{|\al|=m} d^{(k)}_{m,\al} (-\ld^{(k)}_m)X^{(k)}_{m,\al}=&\suml_{m=0}^\infty \suml_{|\al|=m} d^{(k)}_{m,\al} L_k^{(i_0,\ldots,i_k)}X^{(k)}_{m,\al}\\
=&L_k^{(i_0,\ldots,i_k)} X\\
=&-\ld^{(k)}_j \suml_{m=0}^\infty \suml_{|\al|=m} d^{(k)}_{m,\al} X^{(k)}_{m,\al}.
\end{split}
\end{equation*}
For any $i\ne j$, $|\beta|=i$, multiplying  by $w_k X^{(k)}_{i,\beta}$ and then integrating on $\overline{V}_n$ we have
\begin{equation*}
\begin{split}
& \suml_{|\al|=i} d^{(k)}_{i,\al} \ld^{(k)}_i\ple X^{(k)}_{i,\al},w_k^{(i_0,\ldots,i_k)} X^{(k)}_{i,\beta}\pri = \suml_{|\al|=i} d^{(k)}_{i,\al} \ld^{(k)}_j \ple X^{(k)}_{i,\al},w_k^{(i_0,\ldots,i_k)} X^{(k)}_{i,\beta}\pri, \forall i\ne j, |\beta|=i,\\
\Rightarrow & \ple X^{(k)}_{i,\al},w_k^{(i_0,\ldots,i_k)} X^{(k)}_{i,\beta}\pri_{\beta,\al} (d^{(k)}_{i,\al} \ld^{(k)}_i)_{\al}=\ple X^{(k)}_{i,\al},w_k^{(i_0,\ldots,i_k)} X^{(k)}_{i,\beta}\pri_{\beta,\al} (d^{(k)}_{i,\al} \ld^{(k)}_j)_{\al}, \forall i\ne j, |\beta|=i,\\
\Rightarrow & d^{(k)}_{i,\al} \ld^{(k)}_i=d^{(k)}_{i,\al} \ld^{(k)}_j,\quad \forall i\ne j, |\beta|=i, \text{ because } det \ple X^{(k)}_{i,\al},w_k^{(i_0,\ldots,i_k)} X^{(k)}_{i,\beta}\pri_{\beta,\al} \ne 0\\
\Rightarrow & d^{(k)}_{i,\al} = 0,\quad \forall i\ne j, |\al|=i, \text{ because } \ld^{(k)}_i \ne \ld^{(k)}_j.
\end{split}
\end{equation*}
It follows that
$$
X=\suml_{|\al|=j}d^{(k)}_{j,\al} X^{(k)}_{j,\al}.
$$
This completes the proof.
\item[(ii)] is obvious.
\end{enumerate}
\end{proof}
\end{prop}

\subsection{Definition of the solution}
We shall now derive the
Fokker-Planck equation as the diffusion limit of the Wright-Fisher
model and our solution concept for this equation. We consider a diploid population of fixed size
$N$ with $n+1$ possible  alleles $A_1,\ldots,A_{n+1},$ at a given
locus. Suppose that the individuals in the population are monoecious,
that there are no selective differences between these alleles and  no
mutations. There are $2N$ alleles in the population in any generation,
so it is sufficient to focus on the number $Y_m=(Y_m^1,\ldots,Y_m^n)$
of alleles $A_1,\ldots,A_n$ at generation time $m$. Assume that
$Y_0=i_0=(i_0^1,\ldots,i_0^n)$ and according to the Wright-Fisher
model, the alleles in generation $m+1$ are derived by sampling with
replacement from the alleles of generation $m$. Thus,  the transition probability is
$$
\PP(Y_{m+1}=j|Y_m=i)=\frac{(2N)!}{(j^0)! (j^1)! \ldots (j^n)!} \prod_{k=0}^n \ple \frac{i^k}{2N}\pri^{j^k},
$$
where
$$
i,j \in S_n^{(2N)}=\Bigg\{i=(i^1,\ldots,i^n): i^k\in \{0,1,\ldots, 2N\}, \sum_{k=1}^n i^k \le 2N\Bigg\}
$$
and
$$
i^0=2N-|i|=2N-i^1-\ldots-i^n;\quad\quad j^0=2N-|j|=2N-j^1-\ldots-j^n.
$$
After rescaling
$$
t=\frac{m}{2N}, \: \: X_t=\frac{Y_t}{2N},
$$
we have a discrete Markov chain $X_t$ valued in $\ble 0,
\frac{1}{2N},\ldots,1\bri^n$ with $t=1$ now corresponding to $2N$
generations. It is easy to see that
\bel{cond:variable}
\begin{split}
X_0=&p=\frac{i_0}{2N},\\
\E(\delta X^i_t)=&0,\\
\E(\delta X^i_t. \delta X^j_t)=&(X^i_t)(\de_{ij}-X^j_t),\\
\E(\delta X_t)^\al=&(\delta t) \text{ for }|\al|\ge 3.
\end{split}
\qe
We now denote by $m_\al(t)$ the $\al^{th}-$moment of the distribution
about zero at the $t^{th}$ generation, i.e.,
$$
m_\al(t)=\E(X_t)^{\al}
$$
Then
$$
m_\al(t+1)=\E(X_t+\delta X_t)^{\al}
$$
Expanding the right hand side and noting (\ref{cond:variable}) we
obtain the following recursion formula, under the assumption that the
population number $N$ is sufficiently large to neglect terms of order
$\frac{1}{N^2}$ and higher,
\be
m_\al(t+1)=\ble 1-\frac{|\al|(|\al|-1)}{2}\bri m_\al(t)+\suml_{i=1}^n \frac{\al_i (\al_i-1)}{2}m_{\al-e_i}(t)
\qe
Under this assumption, the moments change very slowly per generation
and we can replace this system of difference equations by a system of differential equations:
\bel{cond:moment}
\dot{m}_\al(t)=-\frac{|\al|(|\al|-1)}{2} m_\al(t)+\suml_{i=1}^n \frac{\al_i (\al_i-1)}{2}m_{\al-e_i}(t).
\qe

With the aim to find a continuous process which is a good
approximation for the above discrete process, we should look for a
continuous Markov process $\ble X_t\bri_{t\ge 0}$ valued in $[0,1]^n$
with the same conditions as (\ref{cond:variable}) and
(\ref{cond:moment}). Specially, if we call $u(x,t)$ the probability
density function of this continuous process, the condition
(\ref{cond:variable}) implies (see for example \cite{ewens}, p. 137)
that $u$ is a solution of the Fokker-Planck (Kolmogorov forward) equation
\be
\begin{cases}
u_t&=L_nu \text{ in }V_n\times (0,\infty),\\
u(x,0)&=\de_p(x) \text{ in }V_n;
\end{cases}
\qe
and the condition (\ref{cond:moment}) implies
$$
[u_t,x^\al]_n=\left[u,-\frac{|\al|(|\al|-1)}{2} x^\al+\suml_{i=1}^n \frac{\al_i (\al_i-1)}{2} x^{\al-e_i}\right]_n=[u,L^*(x^\al)]_n, \forall \al,
$$
i.e.,
\be
[u_t,\phi]_n=[u,L_n^*\phi]_n, \forall \phi \in H_n.
\qe

This leads us to the following definition of solutions.
\begin{defn}
We call $u\in H$ a solution of the Fokker-Planck equation associated with the Wright-Fisher model if
\begin{align}
u_t&=L_n u \text{ in }V_n\times (0,\infty),\label{eq:1}\\
u(x,0)&=\de_p(x) \text{ in } V_n;\label{eq:2}\\
[u_t,\phi]_n&=[u,L_n^* \phi]_n,\: \forall \phi\in H_n.\label{eq:3}
\end{align}
\end{defn}

\subsection{The global solution}
In this subsection, we shall construct the solution and prove the
existence as well as the uniqueness of the solution. The process of
finding the solution is as follows: We firstly find the general
solution of the Fokker-Planck equation (\ref{eq:1}) by the separation
of variables method. Then we construct a solution depending on certain
parameters. We then use the conditions of (\ref{eq:2}, \ref{eq:3}) to determine the parameters. Finally, we check the  solution.

{\it Step 1:}
Consider on $V_n$, assume that $u_n(\x,t)=X(\x)T(t)$ is a solution of the Fokker-Planck equation $(\ref{eq:1})$. Then we have
$$
\frac{T_t}{T}=\frac{L_n X}{X}=-\ld
$$
Clearly $\ld$ is a constant which is independent on $T, X$. From the Proposition~(\ref{prop:4}) we obtain the local solution of the equation (\ref{eq:1}) of the form
\begin{equation*}
u_n(\x,t)=\suml_{m=0}^\infty \suml_{|\al|=m}c^{(n)}_{m,\al} X^{(n)}_{m,\al}(\x) e^{-\ld^{(n)}_m t},
\end{equation*}
where
$$
\ld_m^{(n)}=\frac{(n+m)(n+m+1)}{2}
$$
is the eigenvalue of $L_n$ and
$$
X^{(n)}_{m,\al}(\x),\quad |\al|=m
$$
are the corresponding eigenvectors of $L_n$.

For $m\ge 0, |\beta|=m$, we conclude from Proposition~(\ref{prop:2}) that
$$
L_n^*\Big(w_n X^{(n)}_{m,\beta}\Big)=-\ld_m^{(n)}w_n X^{(n)}_{m,\beta}.
$$
It follows that
\begin{align*}
[u_t, w_n X^{(n)}_{m,\beta}]_n &= \Big[u,L_n^*\Big(w_n X^{(n)}_{m,\beta}\Big)\Big]_n\quad \text{(the moment condition)}\\
&=-\ld_m^{(n)}\Big[u,w_n X^{(n)}_{m,\beta}\Big]_n.
\end{align*}
Therefore
\begin{align*}
[u, w_n X^{(n)}_{m,\beta}]_n &= [u(\cdot,0), w_n X^{(n)}_{m,\beta}]_n e^{ -\ld_m^{(n)} t}\\
&= w_n(\p) X^{(n)}_{m,\beta}(\p) e^{ -\ld_m^{(n)} t}.
\end{align*}
Thus,
\begin{align*}
w_n(\p) X^{(n)}_{m,\beta}(\p) e^{ -\ld_m^{(n)} t} &= [u, w_n X^{(n)}_{m,\beta}]_n\\
&=(u_n, w_n X^{(n)}_{m,\beta})_n\quad \text{(because $w_n$ vanishes on boundary)}\\
&=\sum_{|\al|=m} c^{(n)}_{m,\al} (X^{(n)}_{m,\al},w_n X^{(n)}_{m,\beta})_n e^{-\ld^{(n)}_m t}.
\end{align*}
It follows that
$$
\Big(c^{(n)}_{m,\al}\Big)_{\al}= \Bigg[\Bigg((X^{(n)}_{m,\al},w_n X^{(n)}_{m,\beta})_n\Bigg)_{\al,\beta}\Bigg]^{-1}\Bigg(w_n(\p) X^{(n)}_{m,\beta}(\p)\Bigg)_\beta.
$$

{\it Step 2:}
The solution $u\in H$ satisfying (\ref{eq:1}) will be found in the following form
\be
\begin{split}
u(\x,t)=\suml_{k=1}^n u_k(\x,t) \chi_{V_k}(x) + \suml_{i=0}^n u^i_0(\x,t) \de_{e^i}(\x).
\end{split}
\qe

We use the condition (\ref{eq:3}) to obtain gradually values of $u_k, \: k=n-1,\ldots,0$. In fact, assume that we want to calculate $u^{(0,\ldots,n-1)}_{n-1}(x^1,\cdots,x^{n-1},0,t)$.

We note that, if we choose
$$
\phi(\x)=x^1\cdots x^n X^{(n-1)}_{k,\beta}(x^1,\ldots,x^{n-1}), \quad |\beta|=k.
$$
then $\phi(\x)$ vanishes on faces of dimension at most $n-1$ except the face $V_{n-1}^{0,\ldots,n-1}$. Therefore, the expectation of $\phi$ will be
\begin{align*}
[u,\phi]_n=(u_n,\phi)_n+(u^{(0,\ldots,n-1)}_{n-1}, \phi)_{n-1}.
\end{align*}

The left hand side can be calculated easily by the condition (\ref{eq:3})
\bel{eq:balance}
[u_t,\phi]_n= [u,L_n^*(\phi)]_n=-\ld^{(n-1)}_k [u,\phi]_n.
\qe
It follows that
$$
[u,\phi]_n = \phi(\p) e^{-\ld^{(n-1)}_k t}.
$$

The first part of the right hand side is known as
$$
(u_n,\phi)_n= \sum_{m,\al} c^{(n)}_{m,\al}\Bigg( \int_{V_n} X^{(n)}_{m,\al}(\x)\phi(\x)d\x\Bigg) e^{-\ld^{(n)}_m t}.
$$

Therefore we can expand $u^{(0,\ldots,n-1)}_{n-1}(x^1,\cdots,x^{n-1},0,t)$ as follows
\begin{align*}
u^{(0,\ldots,n-1)}_{n-1}(x^1,\cdots,x^{n-1},0,t)&= \sum_{m\ge 0} c^{(n-1)}_m(\x) e^{-\ld^{(n-1)}_m t}\\
&=\sum_{m\ge 0}\sum_{l\ge 0}\sum_{|\al|=l} c^{(n-1)}_{m,l,\al}X^{(n-1)}_{l,\al}(x^1,\ldots,x^{n-1}) e^{-\ld^{(n-1)}_m t}.
\end{align*}

Put this formula into Equation~(\ref{eq:balance}) we will obtain all the coefficients $c^{(n-1)}_{m,l,\al}$. It means that we will obtain $u^{(0,\ldots,n-1)}_{n-1}(x^1,\cdots,x^{n-1},0,t)$.
Similarly we will obtain $u_{n-1}$. And finally we will obtain all $u_k,\quad k=n-1,\ldots,0$. It means we obtain the global solution in form
\be
\begin{split}
u(\x,t)=&\suml_{k=1}^n u_k \chi_{V_k}(\x) +\suml_{i=0}^n u_0^i(\x,t) \delta_{e_i}(\x).\\
=&\suml_{k=1}^n\suml_{m \ge 0}\sum_{l\ge 0} \suml_{|\al|=l}c^{(k)}_{m,l,\al}  X^{(k)}_{l,\al}(\x) e^{-\ld^{(k)}_m t} \chi_{V_k}(\x)  +\suml_{i=0}^n u_0^i(\x,t) \delta_{e_i}(\x).
\end{split}
\qe
It is not difficult to show that $u$ is a solution of the Fokker-Planck equation associated with WF model.

{\it Step 3:}
We can easily see that this solution is unique. In fact, assume that $u_1,u_2$ are two solutions of the Fokker- Planck equation associated with WF model. Then $u=u_1-u_2$ will satisfy

\begin{align*}
u_t&=L_n u \text{ in } V_n \times (0,\infty),\\
u(x,0)&=0 \text{ in } \overline{V}_n;\\
[u_t,\phi]_n&=[u,L^* \phi]_n,\: \forall \phi\in H_n.
\end{align*}
It follows that
\begin{align*}
[u_t,1]_n&=[u,L_n^*(1)]_n=0,\\
[u_t,x^i]_n&=[u,L_n^*(x^i)]_n=0,\\
[u_t,w_k^{(i_0,\ldots, i_k)} X^{(k)}_{j,\al}\chi_{V_k^{(i_0,\ldots, i_k)}}]_n
&=[u,L_n^*(w_k^{(i_0,\ldots, i_k)} X^{(k)}_{j,\al}\chi_{V_k^{(i_0,\ldots, i_k)}})]_n\\
&=[u,L_k^*(w_k^{(i_0,\ldots, i_k)} X^{(k)}_{j,\al}\chi_{V_k^{(i_0,\ldots, i_k)}})]_n\\
&=-\ld^{(k)}_{j} [u,w_k^{(i_0,\ldots, i_k)} X^{(k)}_{j,\al}\chi_{V_k^{(i_0,\ldots, i_k)}}]_n.
\end{align*}
Therefore
\begin{align*}
[u,1]_n&=[u(\cdot, 0),1]_n=0,\\
[u,x^i]_n&=[u(\cdot,0),x^i]_n=0,\\
[u,w_k^{(i_0,\ldots, i_k)} X^{(k)}_{j,\al}\chi_{V_k^{(i_0,\ldots, i_k)}}]_n&=[u(\cdot,0),w_k^{(i_0,\ldots, i_k)} X^{(k)}_{j,\al}\chi_{V_k^{(i_0,\ldots, i_k)}}]_ne^{-\ld^{(k)}_j t}=0.
\end{align*}
Since $\ble 1,\ble x^i\bri_i,\{w_k^{(i_0,\ldots, i_k)} X^{(k)}_{j,\al}\chi_{{V_k}^{(i_0,\ldots, i_k)}}\}_{1\le k\le n,(i_0,\ldots,i_k)\in I_k,j\ge 0,|\al|=j} \bri$ is also a basis of $H_n$ it follows that $u=0 \in H$.

In conclusion, we have established
\begin{thm}\label{thm1}
The Fokker Planck equation associated with the Wright-Fisher model with $(n+1)-$alleles possesses the unique solution
\be
\begin{split}
u(\x,t)=&\suml_{k=1}^n u_k \chi_{V_k}(\x) +\suml_{i=0}^n u_0^i(\x,t) \delta_{e_i}(\x).\\
=&\suml_{k=1}^n\suml_{m \ge 0}\sum_{l\ge 0} \suml_{|\al|=l}c^{(k)}_{m,l,\al}  X^{(k)}_{l,\al}(\x) e^{-\ld^{(k)}_m t} \chi_{V_k}(\x)  +\suml_{i=0}^n u_0^i(\x,t) \delta_{e_i}(\x).
\end{split}
\qe
\end{thm}

\begin{example}
To illustrate this process, we consider the case of three alleles.

We will construct the global solution for the problem
$$
\begin{cases}
\frac{\partial u}{\partial t} &= L_2 u,\quad \text{ in } V_2\times (0,\infty),\\
u(\x,0)&=\de_{\p}(\x),\quad \x\in V_2,\\
[u_t,\phi]_2&=[u,L_2^* \phi]_2, \quad \text{ for all } \phi \in H_2,
\end{cases}
$$
where the global solution of the form
$$
u=u_2 \chi_{V_2}+ u_1^{0,1} \chi_{V_{1}^{0,1}}+u_1^{0,2} \chi_{V_{1}^{0,2}}+u_1^{0,0} \chi_{V_{1}^{0,0}}
+ u_0^1 \chi_{V_0^1}+u_0^2 \chi_{V_0^2}+u_0^0 \chi_{V_0^0}.
$$
and the product is
\begin{align*}
[u,\phi]_2=&\int_{V_2}u_2 \phi_{|V_2}d\x + \int_{0}^1 u_1^{0,1}(x^1,0,t)\phi(x^1,0)dx^1+ \int_{0}^1 u_1^{0,2}(0,x^2,t)\phi(0,x^2)dx^2\\
&+ \frac{1}{\sqrt 2}\int_{0}^1 u_1^{1,2}(x^1,1-x^1,t)\phi(x^1,1-x^1)dx^1\\
&+u_0^1(1,0,t)\phi(1,0)+u_0^2(0,1,t)\phi(0,1)+u_0^0(0,0,t)\phi(0,0).
\end{align*}
Step 1: We find out the local solution $u_2$ as follows
$$
u_2(\x,t)=\sum_{m\ge 0}\sum_{\al^1+\al^2=m} c^{(2)}_{m,\al^1,\al^2} X^{(2)}_{m,\al^1,\al^2}(\x)e^{-\ld_m^{(2)}t}.
$$
To define coefficients $c^{(2)}_{m,\al^1,\al^2}$ we use the initial condition and the orthogonality of eigenvectors $X^{(2)}_{m,\al^1,\al^2}$
\begin{align*}
w_2(\p)X^{(2)}_{m,\beta^1,\beta^2}(\p)
&=\big[u(0),w_2 X^{(2)}_{m,\beta^1,\beta^2}\big]_2\\
&=\big(u_2(0),w_2 X^{(2)}_{m,\beta^1,\beta^2}\big)_2\quad \text{ because $w_2$ vanishes on the boundary}\\
&=\sum_{\al^1+\al^2=m} c^{(2)}_{m,\al^1,\al^2} \Big(X^{(2)}_{m,\al^1,\al^2},w_2X^{(2)}_{m,\beta^1,\beta^2}\Big)\quad \text{for all } \beta^1+\beta^2=m.
\end{align*}
Because the matrix
$$
\Big(X^{(2)}_{m,\al^1,\al^2},w_2X^{(2)}_{m,\beta^1,\beta^2}\Big)_{(\al^1,\al^2), (\beta^1,\beta^2)}
$$
is positive definite then we have unique values of $c^{(2)}_{m,\al^1,\al^2}$. It follows that we have a unique local solution $u_2$.

Step 2: We will use the moment condition to define all other coefficients of the global solution.

Firstly, we define the coefficients of $u_1^{1,2}$ as follows
\begin{align}
u_1^{1,2}(x^1,1-x^1,t)&=\sum_{m\ge 0}c_{m}(x^1) e^{-\ld^{(1)}_m t}\\
&=\sum_{m,l\ge 0}c_{m,l} X^{(1)}_l(x^1) e^{-\ld^{(1)}_m t}.
\end{align}
We note that
$$
L_2^* \Big(x^1 x^2 X_k^{(1)}(x^1)\Big)=-\ld^{(1)}_k  x^1 x^2 X_k^{(1)}(x^1).
$$
Therefore
$$
\big[u_t, x^1 x^2 X_k^{(1)}(x^1)\big]_2=\Big[u, L^*_2 \Big(x^1 x^2 X_k^{(1)}(x^1)\Big)\Big]_2=-\ld^{(1)}_k  [u,x^1 x^2 X_k^{(1)}(x^1)]_2.
$$
It follows that
$$
\big[u,x^1 x^2 X_k^{(1)}(x^1)\big]_2=p^1 p^2 X_k^{(1)}(p^1) e^{-\ld_k^{(1)}t}.
$$
Thus we have
\begin{align*}
p^1 p^2 X_k^{(1)}(p^1) e^{-\ld_k^{(1)}t}&=\big[u,x^1 x^2 X_k^{(1)}(x^1)\big]_2\\
&=\Big(u_2,x^1 x^2 X_k^{(1)}(x^1)\Big)_2 + \Big(u_1^{1,2},x^1 (1-x^1) X_k^{(1)}(x^1) \Big)_1\\
&\quad \text{because $x^1 x^2$ vanish on the other boundaries}\\
&=\sum_{m\ge 0}\Bigg(\sum_{|\al|=m} c^{(2)}_{m,\al}\Bigg( \int_{V_2} x^1 x^2 X_{m,\al}^{(2)}(x^1,x^2) X_k^{(1)}(x^1)d\x \Bigg)\Bigg)e^{-\ld^{(2)}_{m}t}\\
&\quad +\sum_{m\ge 0} c_{m,k} \Big(X_k^{(1)},w_1 X_k^{(1)}\Big) e^{-\ld^{(1)}_m t}\\
&\quad \text{because of the orthogonality of $(\cdot,\cdot)_1$ with respect to $w_1$}\\
&=\sum_{m\ge 0}r_m e^{-\ld^{(2)}_{m}t}+\sum_{m\ge 0} c_{m,k}d_k e^{-\ld^{(1)}_m t}\\
\end{align*}

By equating of coefficients of $e^{\al t}$ we obtain $u_1^{1,2}$. Similarly we obtain $u_1$. Then, we define the coefficients of $u_0^1$ from the $1-$th moment.

Note that when $\phi = x^i$, $L_2^*(\phi)=0$, therefore $[u_t,\phi]_2=0$ or
$$
[u,x^i]_2=[u(0),x^i]=p^i.
$$
It follows that
\begin{align*}
p^1=[u,x^1]= (u_2,x^1)_2 + (u_1^{0,1},x^1)_1+(u_1^{1,2},x^1)_1+ u_0^1(1,0,t).
\end{align*}
Thus we obtain $u_0^1(1,0,t)$. Similarly we have all $u_0$. Therefore we obtain the global solution $u$.

It is easy to check that $u$ is a global solution. To prove the uniqueness we proceed as follows
Assume that $u$ is the difference of any two global solutions, i.e. $u$ satisfies
$$
\begin{cases}
u_t&=L_2u, \quad \text{ in } V_2\times (0,\infty),\\
u(\x,0)&=0,\quad \text{ in } V_2\\
[u_t,\phi]_2&=[u,L_2^* \phi]_2, \quad \text{ for all } \phi \in H_2.
\end{cases}
$$
We will prove that
\bel{eq:zerosolution}
[u,\phi]_2=0\quad \forall \phi\in H_2.
\qe
In fact,
\begin{align*}
[u_t,1]_2&=[u,L_2^*(1)]_2=0 \Rightarrow [u,1]_2=[u(0),1]_2=0,\\
[u_t,x^i]_2&=[u,L_2^*(x^i)]_2=0 \Rightarrow [u,x^i]_2=[u(0),x^i]_2=0,\\
[u_t,w_1(x^i)X^{(1)}_m(x^i)]_2&=[u,L_2^*(w_1(x^i)X^{(1)}_m(x^i))]_2=-\ld_m^{(1)} [u,w_1(x^i)X^{(1)}_m(x^i)]_2\\ &\Rightarrow [u,w_1(x^i)X^{(1)}_m(x^i)]_2=[u(0),w_1(x^i)X^{(1)}_m(x^i)]_2 e^{-\ld_m^{(1)} t}=0,\\
[u_t,w_2(x^1,x^2)X^{(2)}_{m,\al}(x^1,x^2)]_2&=[u,L_2^*(w_2(x^1,x^2)X^{(2)}_{m,\al}(x^1,x^2))]_2=-\ld_m^{(2)} [u,w_2(x^1,x^2)X^{(2)}_{m,\al}(x^1,x^2)]_2\\
&\Rightarrow [u,w_2(x^1,x^2)X^{(2)}_{m,\al}(x^1,x^2)]_2=[u(0),w_2(x^1,x^2)X^{(2)}_{m,\al}(x^1,x^2)]_2 e^{-\ld_m^{(2)} t}=0.
\end{align*}
We need only to prove that Eq.~(\ref{eq:zerosolution}) holds for all
$$
\phi(x^1,x^2)=(x^1)^m(x^2)^n,\quad \forall m,n\ge 0.
$$
\begin{enumerate}
\item If $n=0, m\ge 0$, we see that $\phi$ can be generated from $\{1,x^1,w_1(x^1)X^{(1)}_m(x^1)\}$, therefore $[u,\phi]_2=0$
\item If $m=0, n\ge 0$, we see that $\phi$ can be generated from $\{1,x^2,w_1(x^2)X^{(1)}_m(x^2)\}$, therefore $[u,\phi]_2=0$
\item If $n=1, m\ge 1$, we expand $(x^1)^{m-1}$ by
$$
(x^1)^{m-1}=\sum_{k\ge 0} c_k X_k^{(1)}(x^1).
$$
Note that
$$
L_2^*\Big(x^1 x^2 X_k^{(1)}(x^1)\Big)=-\ld_k^{(1)}x^1 x^2 X_k^{(1)}(x^1)
$$
Therefore
$$
[u_t, x^1 x^2 X_k^{(1)}(x^1)]_2=[u, L_2^*\Big(x^1 x^2 X_k^{(1)}(x^1)\Big)]_2=-\ld_k^{(1)} [u,x^1 x^2 X_k^{(1)}(x^1)]_2.
$$
It follows that
$$
[u,x^1 x^2 X_k^{(1)}(x^1)]_2=[u(0),x^1 x^2 X_k^{(1)}(x^1)]_2 e^{-\ld^{(1)}_k}=0.
$$
Therefore
$$
[u,\phi]_2=\sum_{k\ge 0} c_k [u,x^1 x^2 X_k^{(1)}(x^1)]_2=0.
$$
\item If $n\ge 2, m\ge 1$ we use the inductive method in $n$. We have
\begin{align*}
(x^1)^m(x^2)^n &= x^1x^2(x^1+x^2-1)(x^1)^{m-1}(x^2)^{n-2}+(x^1)^m(1-x^1)(x^2)^{n-1}\\
&=-w_2(x^1,x^2)(x^1)^{m-1}(x^2)^{n-2}+(x^1)^m(1-x^1)(x^2)^{n-1}.
\end{align*}
In the assumption of induction, we have
$$
[u, (x^1)^m(1-x^1)(x^2)^{n-1}]_2=0
$$
Then, we expand $(x^1)^{m-1}(x^2)^{n-2}$ by
$$
(x^1)^{m-1}(x^2)^{n-2}=\sum_{m,\al}c^{(2)}_{m,\al}X^{(2)}_{m,\al}(x^1,x^2).
$$
Therefore
$$
[u,w_2(x^1,x^2)(x^1)^{m-1}(x^2)^{n-2}]_2=\sum_{m,\al}c^{(2)}_{m,\al}[u,w_2(x^1,x^2)X^{(2)}_{m,\al}(x^1,x^2)]_2=0.
$$
It follows that $[u,(x^1)^m(x^2)^n]_2=0$.

Thus, $u=0$.
\end{enumerate}
\end{example}

\section{Applications}
In this section, we present some applications of our global solution
to the evolution of the process $(X_t)_{t\ge 0}$ such as the
expectation and the second moment of the absorption time, the probability distribution of the absorption time for having $k+1$ alleles, the probability of having  exactly $k+1$ alleles, the $\al^{th}$ moments, the probability of heterogeneity, and the rate of loss of one allele in a population having $k+1$ alleles. We refer to read \cite{ewens}, \cite{kimura1}, \cite{kimura2}, \cite{kimura3}, \cite{littler1}, \cite {littler2}, etc., which of these results are already known.

\subsection{The absorption time for having $(k+1)$ alleles}
We denote by $T^{k+1}_{n+1}(p)=\inf \ble{t>0: X_t\in
  \overline{V}_k}|X_0=p \bri$ the first time when the  population has
(at most) $k+1$ alleles. $T^{k+1}_{n+1}(p)$ is a continuous random
variable valued in $[0,\infty)$ and we denote by $\phi(t,p)$ its
probability density function. It is easy to see that $\overline{V}_k$
is invariant under the process $(X_t)_{t \ge 0}$, i.e. if $X_s \in
\overline{V}_k$ then $X_t\in \overline{V}_k$ for all $t\ge s$ (once an
allele is lost from the population, it can never again be recovered). We have the equality
$$
\PP(T^{k+1}_{n+1}(p)\le t)=\PP(X_{t}\in \overline{V}_k|X_0=p)=\int_{\overline{V}_k}u(x,p,t)d\mu(x).
$$
It follows that
$$
\phi(t,p)=\int_{\overline{V}_k}\frac{\partial}{\partial t} u(x,p,t)d\mu(x)
$$
Therefore the expectation for the absorption time of having $k+1$ alleles is (see also \cite{ewens}, p. 194)
\begin{equation*}
\begin{split}
\E(T^{k+1}_{n+1}(p))=&\int_0^\infty t\phi(t,p)dt\\
=&\int_{\overline{V}_k}\int_0^\infty t\frac{\partial}{\partial t} u(x,p,t)dtd\mu(x)\\
=&\suml_{j=1}^k \suml_{(i_0,\ldots,i_j)\in I_j} \suml_{m\ge 0} \suml_{|\alpha|=m} c_{m,\alpha}^{(j)} \int_{{V}^{(i_0,\ldots,i_j)}_j} X^{(j)}_{m,\alpha}(x)\ple \int_0^\infty t\frac{\partial}{\partial t} e^{-\ld_m^{(j)}t}dt\pri d\mu^{(i_0,\ldots,i_j)}_j(x)\\
&+\suml_{i=0}^n \suml_{k=1}^n\suml_{m \ge 0} \suml_{|\al|=m}c^{(k)}_{m,\al}a^{(k)}_{m,\al,i}\ple \int_0^\infty t\frac{\partial}{\partial t} e^{-\ld^{(k)}_m t}dt\pri,\\
=&\suml_{j=1}^k \suml_{(i_0,\ldots,i_j)\in I_j} \suml_{m\ge 0} \suml_{|\alpha|=m} c_{m,\alpha}^{(j)} \int_{{V}^{(i_0,\ldots,i_j)}_j} X^{(j)}_{m,\alpha}(x)\ple -\frac{1}{\ld_m^{(j)}}\pri d\mu^{(i_0,\ldots,i_j)}_j(x)\\
&+\suml_{i=0}^n \suml_{k=1}^n\suml_{m \ge 0} \suml_{|\al|=m}c^{(k)}_{m,\al}a^{(k)}_{m,\al,i}\ple -\frac{1}{\ld_m^{(k)}}\pri.
\end{split}
\end{equation*}

and the second moment of this absorption time  is (see also \cite{littler2})
\begin{equation*}
\begin{split}
\E(T^{k+1}_{n+1}(p))^2=&\int_0^\infty t^2\phi(t,p)dt\\
=&\int_{\overline{V}_k}\int_0^\infty t^2\frac{\partial}{\partial t} u(x,p,t)dtd\mu(x)\\
=&\suml_{j=1}^k \suml_{(i_0,\ldots,i_j)\in I_j} \suml_{m\ge 0} \suml_{|\alpha|=m} c_{m,\alpha}^{(j)} \int_{{V}^{(i_0,\ldots,i_j)}_j} X^{(j)}_{m,\alpha}(x)\ple \int_0^\infty t^2\frac{\partial}{\partial t} e^{-\ld_m^{(j)}t}dt\pri d\mu^{(i_0,\ldots,i_j)}_j(x)\\
&+\suml_{i=0}^n \suml_{k=1}^n\suml_{m \ge 0} \suml_{|\al|=m}c^{(k)}_{m,\al}a^{(k)}_{m,\al,i}\ple \int_0^\infty t^2\frac{\partial}{\partial t} e^{-\ld^{(k)}_m t}dt\pri,\\
=&\suml_{j=1}^k \suml_{(i_0,\ldots,i_j)\in I_j} \suml_{m\ge 0} \suml_{|\alpha|=m} c_{m,\alpha}^{(j)} \int_{{V}^{(i_0,\ldots,i_j)}_j} X^{(j)}_{m,\alpha}(x)\ple -\frac{2}{(\ld_m^{(j)})^2}\pri d\mu^{(i_0,\ldots,i_j)}_j(x)\\
&+\suml_{i=0}^n \suml_{k=1}^n\suml_{m \ge 0} \suml_{|\al|=m}c^{(k)}_{m,\al}a^{(k)}_{m,\al,i}\ple -\frac{2}{(\ld_m^{(k)})^2}\pri.
\end{split}
\end{equation*}

\subsection{The probability distribution of the absorption time for having
  $k+1$ alleles}
We note that $X_{T^{k+1}_{n+1}(p)}$ is a random variable valued in
$\overline{V_k}$. We consider the probability that this random
variable takes its value in $V_k^{(i_0,\ldots,i_k)}$, i.e., the
probability of the population at the first time having at most
$k+1$ alleles to consist precisely  of the $k+1$ alleles $\{A_{i_0},\ldots,A_{i_k}\}$.
Let $g_k$ be a function of $k$ variables defined inductively by
\begin{equation*}
\begin{split}
g_1(p^1)=&p^1;\\
g_2(p^1,p^2)=&\frac{p^1}{1-p^2}g_1(p^2)+\frac{p^2}{1-p^1}g_1(p^1);\\
g_{k+1}(p^1,\ldots,p^{k+1})=&\suml_{i=1}^{k+1}\frac{p^i}{1-\suml_{j\ne i}p^j}g_k(p^1,\ldots,p^{i-1},p^{i+1},\ldots,p^{k+1})
\end{split}
\end{equation*}
Then we shall have
\begin{thm}
$$
\PP \left(X_{T^{k+1}_{n+1}(p)} \in \overline{V^{(i_0,\ldots,i_k)}_k}\right)=g_{k+1}(p^{i_0},\ldots,p^{i_k}).
$$
\end{thm}
\begin{proof}
Method 1: By proving that
$$
\PP \left(X_{T^{k+1}_{n+1}(p)} \in \overline{V^{(i_0,\ldots,i_k)}_k}|X_{T^{k}_{n+1}(p)} \in \overline{V^{(i_1,\ldots,i_k)}_k}\right)=\frac{p^{i_0}}{1-p^{i_1}-\ldots-p^{i_k}}.
$$
and elementary combinatorial arguments, we have immediately the result (see also \cite{littler1})

Method 2: By proving that it is the unique solution of the classical Dirichlet problem
\begin{equation*}
\begin{cases}
(L_k^{(i_0,\ldots,i_k)})^* v(p)&=0 \text{ in } V_k\\
\lim\limits_{p\to q} v(p)&=1, q\in V_k^{(i_0,\ldots,i_k)},\\
\lim\limits_{p\to q} v(p)&=0, q\in \partial V_k \backslash V_k^{(i_0,\ldots,i_k)} \backslash V_{k-1}.
\end{cases}
\end{equation*}
\end{proof}

\subsection{The probability of having exactly $k+1$ alleles}
The probability of having only $1$ allele $A_i$ (allele $A_i$ is fix) is
\begin{equation*}
\begin{split}
\PP(X_t\in V^{(i)}_0|X_0=\p)&=\intl_{V^{(i)}_{0}}u_0^{(i)}(\x,t)d\mu^{(i)}_0(\x)\\
&=u^{(i)}_0(e_i,t)\\
&= p^i-\suml_{k=1}^n\suml_{m^{(k)} \ge 0} \sum_{l^{(k)}\ge 0} \suml_{|\al^{(k)}|=l^{(k)}}c^{(k)}_{m^{(k)},l^{(k)},\al^{(k)}} \Big(x^i, X^{(k)}_{l^{(k)},\al^{(k)}}\Big)_k e^{-\ld^{(k)}_{m^{(k)}} t}.
\end{split}
\end{equation*}
The probability of having exactly $(k+1)$ allele $\{A_0,\ldots,A_k\}$ (the coexistence probability of alleles $\{A_0,\ldots,A_k\}$) is (see also \cite{littler2, kimura2})
\begin{equation*}
\begin{split}
\PP(X_t\in V^{(i_0,\ldots,i_k)}_k|X_0=\p)&=\intl_{V^{(i_0,\ldots,i_k)}_k}u^{(i_0,\ldots,i_k)}_k(\x,t)d\mu^{(i_0,\ldots,i_k)}_k(\x)\\
&=\suml_{m\ge 0}\sum_{l\ge 0}\suml_{|\alpha|=l}c^{(k)}_{m,l,\al}\ple \intl_{V^{(i_0,\ldots,i_k)}_k} X^{(k)}_{m,\al}(\x)d\mu^{(i_0,\ldots,i_k)}_k(\x)\pri e^{-\ld_m^{(k)}t}.
\end{split}
\end{equation*}

\subsection{The $\al^{th}$ moments}
The $\al^{th}$-moments are (see also \cite{kimura1, kimura2, kimura3})
\begin{equation*}
\begin{split}
m_\alpha(t)=& [u,\x^\alpha]_n\\
=&\intl_{\overline{V_n}}x^\alpha u(\x,t)d\mu(\x)\\
=&\suml_{k=0}^n \suml_{(i_0,\ldots,i_k)\in I_k} \intl_{V^{(i_0,\ldots,i_k)}_k}\x^\alpha u_k^{(i_0,\ldots,i_k)}(\x,t)d\mu^{(i_0,\ldots,i_k)}_k(\x).
\end{split}
\end{equation*}

\subsection{The probability of heterogeneity}
The probability of heterogeneity is (see also \cite{kimura2})
\begin{equation*}
\begin{split}
H_t = & (n+1)! \: [u,w_n]_n\\
= & (n+1)! \: (u_n,w_n)_n \quad \text{ (because $w_n$ vanishes on the boundary)}\\
= & (n+1)! \: \Big(\sum_{m\ge 0} \sum_{|\alv|=m} c^{(n)}_{m,\alv} X^{(n)}_{m,\alv} e^{-\ld^{(n)}_{m,\alv}t}, w_n X^{(n)}_{0,\ov}\Big)_n\\
= & (n+1)! \: \Big(c^{(n)}_{0,\ov} X^{(n)}_{0,\ov}, w_n X^{(n)}_{0,\ov}\Big)_n \: e^{-\ld^{(n)}_{0,\ov}t} \quad \text{ (because of the orthogonality of the eigenvectors $X^{(n)}_{m,\alv}$)}\\
= & H_0\: e^{-\frac{(n+1)(n+2)}{2}t}
\end{split}
\end{equation*}

\subsection{The rate of loss of one allele in a population having $k+1$ alleles}
We have the solution of the form
$$
u=\sum_{k=0}^n u_{k}(\x,t) \chi_{V_k}(\x)
$$
The rate of loss of one allele in a population with (k+1) alleles equals the rate of decrease of
$$
u_{k}(\x,t)=\suml_{m \ge 0} \sum_{l\ge 0}\suml_{|\al|=l}c^{(k)}_{m,l,\al} X^{(k)}_{l,\al}(x)\chi_{V_k}(x) e^{-\ld^{(k)}_m t}.
$$
which is $\ld^{(k)}_0=\frac{k(k+1)}{2}$. This means the rate of loss of alleles in the population decreases (see also \cite{kimura2}).

\section*{Conclusion}
We have developed a new global solution concept for the Fokker-Planck
equation associated with the Wright-Fisher model, and we have proved
the existence and uniqueness of this solution (Theorem
\ref{thm1}). From this solution, we can easily read off the properties
of the considered process, like the  absorption time of having $k+1$
alleles, the probability of having exactly $k+1$ alleles, the $\al^{th}$ moments, the probability of heterogeneity, and the rate of loss of one allele in a population having $k+1$ alleles.

\newpage
\bibliographystyle{amsplain}

\begin{thebibliography}{99}

\bibitem{buerger} R.B\"urger, The mathematical theory of selection,
  recombination, and mutation, John Wiley, 2000

\bibitem{ewens}
Warren J. Ewens, {\it Mathematical Population Genetics I. Theoretical Introduction}, Springer-Verlag New York Inc., Interdisciplinary Applied Mathematics, 2nd ed., 2004.

\bibitem{fisher}
Fisher R. A., {\it On the dominance ratio}, Proc. Roy. Soc. Edinb., {\bf 42} (1922), 321-341.

\bibitem{kimura1}
Kimura M., {\it Solution of a Process of Random Genetic Drift with a
Continuous Model}, PNAS--USA, Vol. 41, No. 3, (1955), 144-150.

\bibitem{kimura2}
Kimura M., {\it Random genetic drift in multi-allele locus}, Evolution, \textbf{9} (1955), 419-435.

\bibitem{kimura3}
Kimura M., {\it Random genetic drift in a tri-allelic locus; exact solution with a continuous model}, Biometrics, \textbf{12} (1956), 57-66.

\bibitem{littler1}
Littler R. A., Good A.J., {\it Ages, Extinction times, and First Passage Probabilities for a Multiallele Diffusion Model with Irreversible Mutation}, Theoretical Population Biology, {\textbf 13} (1978), 214-225.

\bibitem{littler2}
Littler R.A., {\it Loss of Variability at One Locus in a Finite Population}, {Mathematical Biosciences}, (\textbf 25) (1975), 151-163.

\bibitem{risken}
Risken, H., {\it The Fokker-Planck-Equation. Methods of Solution and
Applications}, 2nd ed., Berlin etc., Springer-Verlag 1989.

\bibitem{thj1} T.D.Tran, J.Hofrichter, J.Jost, {\it The mathematical structure of the Wright-Fisher model of
  population genetics}

\bibitem{wright1}
S. Wright, {\it Evolution in Mendelian populations}, Genetics, {\bf 16} (1931), 97-159.

\bibitem{wright2}
S. Wright, {\it The differential equation of the distribution of gene frequencies}, Proc. Nat. Acad. Sci., {\bf 31} (1945), 382-389.

\end{thebibliography}

\end{document}